\documentclass[11pt]{article}

\usepackage[margin=1in]{geometry}
\usepackage{amsmath,amssymb,amsthm,mathtools}
\usepackage{booktabs,tabularx,array,longtable,multirow}
\usepackage{enumitem}
\usepackage{placeins}
\usepackage{microtype}
\usepackage{lineno}
\usepackage{xcolor}
\usepackage[normalem]{ulem}
\usepackage{hyperref}
\hypersetup{
  colorlinks=false,
  pdfborder={0 0 1},
  citebordercolor={0.15 0.45 0.75},
  linkbordercolor={1 1 1},
  urlbordercolor={1 1 1}
}
\usepackage[capitalize,noabbrev]{cleveref}

\newtheorem{theorem}{Theorem}[section]
\newtheorem{proposition}[theorem]{Proposition}
\newtheorem{lemma}[theorem]{Lemma}
\newtheorem{corollary}[theorem]{Corollary}
\theoremstyle{definition}

\theoremstyle{remark}

\crefname{theorem}{theorem}{theorems}
\crefname{proposition}{proposition}{propositions}
\crefname{lemma}{lemma}{lemmas}
\crefname{corollary}{corollary}{corollaries}
\crefname{definition}{definition}{definitions}
\crefname{example}{example}{examples}
\crefname{remark}{remark}{remarks}

\newcommand{\R}{\mathbb R}
\newcommand{\Smat}{\mathbb S}
\newcommand{\Splus}{\mathbb S_+}
\newcommand{\one}{\mathbf 1}
\newcommand{\cP}{\mathcal P}
\newcommand{\ip}[2]{\left\langle #1,#2\right\rangle}

\newcommand{\tr}{\operatorname{tr}}

\newcommand{\Diag}{\operatorname{Diag}}
\newcommand{\arrow}{\operatorname{arrow}}
\newcommand{\range}{\operatorname{range}}
\newcommand{\nullspace}{\operatorname{null}}
\newcommand{\sym}{\operatorname{sym}}

\title{Exact Convex Reformulations of Binary Quadratic Programs\\
from Facially Reduced Semidefinite Relaxations}
\author{Hao Hu\thanks{\raggedright School of Mathematical and Statistical Sciences,
Clemson University, Clemson, SC, USA.
Email: \href{mailto:hhu2@clemson.edu}{\texttt{hhu2@clemson.edu}}.}
\and Mingming Xu\thanks{\raggedright School of Mathematical and Statistical Sciences,
Clemson University, Clemson, SC, USA.
Email: \href{mailto:mingmix@g.clemson.edu}{\texttt{mingmix@g.clemson.edu}}.}}
\date{}

\begin{document}
\maketitle

\begin{abstract}
We develop MIQCR--FR, a framework for exact convex reformulation from facially reduced
semidefinite relaxations of equality-constrained binary quadratic problems.
Standard mixed-integer quadratic convex reformulation (MIQCR) uses dual
multipliers of the original relaxation. Facial reduction (FR) reduces the
semidefinite matrix order, but
recovery of the original dual multipliers is not guaranteed in general.
For this class of problems, we identify how retaining
the linear equalities and their products with each binary variable enables
this recovery. An established dual recovery formula gives
an explicit extension of every feasible solution of the reduced dual to
the original dual, preserving its bound. For approximate multiplier
estimates, we give a spectral shift restoring dual feasibility and two
constructions optimizing different bounds within a restricted multiplier
family. Maximizing the dual bound yields a trust-region subproblem;
maximizing the continuous MIQCR bound requires only a reduced
eigenvalue calculation. The latter construction enforces convexity directly,
rather than full dual feasibility. We implement it using multiplier estimates
from the alternating direction method of multipliers (ADMM).
On 431 instances from six problem
families, the complete method certifies 296 optima, compared with 156 for
the conic-bundle method MIQCR-CB and 187 for direct Gurobi under a common
time budget.
\end{abstract}

\noindent\textbf{Keywords:}
binary quadratic programming; quadratic convex reformulation; semidefinite
programming; facial reduction; ADMM; quadratic assignment problem

\section{Introduction}

Convex reformulation offers a way to combine strong semidefinite bounds
with the integer-search capabilities of general-purpose optimization
solvers. Its computational value depends on both stages: the cost of
constructing the reformulation and the cost of solving it. This paper
establishes how the original linear equalities, together with their products
with each binary variable, connect facially reduced semidefinite optimization
with exact convex reformulation, through explicit
dual recovery and constructions for approximate multiplier estimates.
Computational experiments show that the resulting method can improve both
optimality certification and feasible-solution quality compared with solving
the original formulation directly.

We consider the equality-constrained binary quadratic problem (BQP)
\begin{equation}\label{eq:bqp}
  \min\left\{
      g(x):=x^TQx+c^Tx:
      Ax=b,\ x\in\{0,1\}^n
  \right\},
\end{equation}
where $Q\in\Smat^n$, $c\in\R^n$, $A\in\R^{m\times n}$, and
$b\in\R^m$. Redundant rows of $A$ are removed, so $A$ has full row
rank. This setting includes quadratic assignment, semi-assignment,
quadratic shortest path (QSPP), cycle cover, and cardinality-constrained
graph problems. Our computational
focus is on nonconvex objectives for which stronger relaxation bounds can
justify the cost of reformulation.

Quadratic convex reformulation (QCR), developed by Billionnet, Elloumi,
and Plateau~\cite{BillionnetEP09}, follows a two-phase procedure.
Phase~1 solves a semidefinite programming (SDP) relaxation to construct a
convex quadratic objective that agrees with $g$ at every feasible binary
point. Phase~2 solves the resulting convex binary quadratic problem with a
mixed-integer solver. The SDP is solved during preprocessing and need not
be solved at the nodes of the subsequent search tree.

Adding entrywise nonnegativity to the lifted positive-semidefinite matrix
gives a doubly nonnegative (DNN) relaxation. Its inequality multipliers,
however, generally yield an underestimating objective rather than one
that agrees with $g$ on the binary feasible set. Mixed-integer quadratic
convex reformulation (MIQCR) restores this agreement by introducing
continuous variables for selected products $x_ix_j$ and imposing linear
constraints that enforce the products when $x$ is binary
\cite{billionnet2012extending}. The resulting convex mixed-integer
quadratic program (MIQP) is exact, but its additional variables can make
Phase~2 more expensive. MIQCR-CB addresses the cost of both phases through
a conic-bundle method and a dynamic working set of
reformulation-linearization technique (RLT) constraints
\cite{billionnet2017conic}.

Equality constraints provide another way to reduce the cost of Phase~1.
The corresponding lifted equality system restricts the feasible matrices to a face of the
positive semidefinite cone. Facial reduction represents the feasible matrices in a
lower-dimensional positive semidefinite cone, reducing the matrix order
without changing the relaxation bound. For example, the alternating direction method of multipliers (ADMM)
can exploit the reduced formulation through a smaller positive-semidefinite
projection and an explicit polyhedral projection. Oliveira, Wolkowicz,
and Xu~\cite{oliveira2018admm} exploit this structure to compute bounds and
heuristic solutions for the quadratic assignment problem (QAP). Integrating
this approach into MIQCR requires converting its multiplier estimates into
an exact convex reformulation for Phase~2.

The standard MIQCR construction uses dual multipliers for the original
semidefinite relaxation, whereas facial reduction does not generally
guarantee their recovery. Permenter and Parrilo study conditions under which
such recovery is possible~\cite{permenter2018partial}. A central observation
of this paper is that the lifted equality system makes these two approaches
compatible. The same
equations that restrict the relaxation to a smaller face also supply the
multiplier freedom needed to extend every feasible point of the reduced dual to
the original formulation without changing its bound.
The dual recovery formula of Hou, Tang, and Toh~\cite{hou2026poly} provides an
explicit realization of this extension. We exploit this
connection to construct exact convex reformulations from optimization in
the reduced space.

For approximate multiplier estimates, dual feasibility can be restored by
an explicit spectral shift or by optimizing the dual bound within a
restricted multiplier family. Dual recovery then gives an exact convex
reformulation. However, full dual feasibility is stronger than the
convexity condition needed for reformulation: only the quadratic
coefficient matrix must be positive semidefinite. Applying the same completion formula directly to this
matrix yields a third construction, which optimizes the continuous MIQCR
bound within the chosen family. This direct construction is used in our
experiments. The framework therefore connects reduced-space optimization
to convex reformulation both through dual recovery and through a weaker
convexity requirement. All three constructions depend on the available
multiplier estimates rather than on the algorithm used to compute them.
We call this framework MIQCR--FR, with FR denoting facial reduction.
Our implementation uses ADMM to solve the reduced relaxation and extracts
the required estimates from its output.

Related QCR
extensions address semicontinuous variables, linear on--off constraints,
binary polynomial programs, and triangle-inequality strengthenings
\cite{WuSunLiZheng17,WuLiJiang19,ElloumiLambertLazare21,Lambert23}.
Recent low-rank augmented-Lagrangian methods also solve DNN and SDP--RLT
relaxations and recover dual information for optimality verification
\cite{hou2026rinnal,hou2026poly}. Our focus is on exact convex reformulations from multiplier estimates for
the reduced relaxation, the bounds attainable within a restricted multiplier
family, and the performance of the complete two-phase method.

Our contributions are as follows.
\begin{enumerate}[leftmargin=*]
  \item We establish that every feasible solution of the dual of the
  facially reduced relaxation can be extended to the original dual without
  changing its objective value, enabling exact convex reformulation from
  the reduced problem. This guarantee follows from retaining the original
  linear equalities and their products with each binary variable, and the
  extension is obtained through an established dual recovery formula.

  \item We develop methods for exact convex reformulation from approximate
  multiplier estimates. An explicit spectral shift restores dual feasibility
  with a computable change in the dual value, and maximizing the dual bound
  within a restricted multiplier family is equivalent to a trust-region
  subproblem. We further show that enforcing convexity directly, under a
  weaker condition than full dual feasibility, yields the strongest
  continuous MIQCR bound in the same family with fixed product variables and McCormick constraints
  through a minimum-eigenvalue calculation.

  \item We implement MIQCR--FR using ADMM and the direct continuous-bound
  construction, and compare it with MIQCR-CB and direct Gurobi on 431
  instances from six problem families. Under a common time budget, MIQCR--FR
  certifies more optima than either comparator and improves
  feasible-solution quality on several families.
\end{enumerate}

The remainder of the paper is organized as follows.
\Cref{sec:dnn-formulations} presents the original and facially reduced DNN
relaxations, their duals, and the standard MIQCR construction.
\Cref{sec:dual-recovery} gives
exact dual recovery and the three constructions for approximate
multipliers.
\Cref{sec:computational-design} presents the two-phase workflow, its ADMM implementation, and
the computational experiments.

\section{DNN relaxations, facial reduction, and MIQCR}
\label{sec:dnn-formulations}

We first state the DNN relaxation and its dual, then recall how a feasible
dual certificate supplies an exact convex MIQCR reformulation. We next
use the lifted equalities to obtain a facially reduced relaxation and its dual.

\subsection{The DNN relaxation and its dual}

To obtain a semidefinite relaxation, we lift the quadratic products into
a matrix variable. This makes the objective linear in the lifted matrix
and expresses the binary identities $x_i^2=x_i$ as linear equations.
For a binary vector $x$, define the rank-one Shor matrix
\begin{equation}\label{eq:rank-one-lift}
  Y:=
  \begin{pmatrix}1\\x\end{pmatrix}
  \begin{pmatrix}1\\x\end{pmatrix}^{\!T}
  =
  \begin{pmatrix}1&x^T\\x&xx^T\end{pmatrix}\succeq0,
\end{equation}
whose entries satisfy $Y_{ii}=Y_{0i}=x_i$ for $i=1,\ldots,n$.
We index its rows and columns by $0,1,\ldots,n$.
The relaxation replaces this rank-one matrix by a general symmetric
matrix variable $Y\in\Smat^{n+1}$. With the objective matrix
\begin{equation}\label{eq:Y-lift}
  C=\begin{pmatrix}
     0 & c^T/2\\
     c/2 & Q
  \end{pmatrix},
\end{equation}
the linear objective $\ip{C}{Y}$ agrees with $g(x)$ on every rank-one
lift~\eqref{eq:rank-one-lift}. We retain the normalization $Y_{00}=1$
and the binary identities through the arrow equations $Y_{0i}=Y_{ii}$.
Together, these constraints are written as $\arrow(Y)=e_0$, where
$e_0=(1,0,\ldots,0)^T$ and
\begin{equation}\label{eq:qcr-affine-operator}
  \arrow:\Smat^{n+1}\to\R^{n+1},
  \qquad
  \arrow(Y)
  :=
  \begin{pmatrix}Y_{00}\\ (Y_{0i}-Y_{ii})_{i=1}^n\end{pmatrix}.
\end{equation}

To incorporate $Ax=b$ and its products with each binary variable,
set $B:=[-b\ \ A]\in\R^{m\times(n+1)}$. Every rank-one
lift~\eqref{eq:rank-one-lift} of a point satisfying $Ax=b$ obeys
\begin{equation}\label{eq:BY0}
  BY=0,
\end{equation}
which we call the \emph{lifted equality system}.
With $Y_{00}=1$, the first column retains the original equalities, and the
remaining columns retain their products with each binary variable after
lifting. These lifted equalities are obtained by the reformulation-linearization
technique of Sherali and Adams~\cite{sherali1990hierarchy}.

We retain these equalities and impose entrywise nonnegativity, without
adding the full McCormick inequalities. Dropping the rank-one restriction
gives the DNN relaxation underlying the method:
\begin{equation}\label{eq:sdp-dnn}
  \min_{Y\in\Smat^{n+1}}\bigl\{\ip{C}{Y}:\arrow(Y)=e_0,\ BY=0,\ Y\ge0,\ Y\succeq0\bigr\}.
\end{equation}
Every feasible binary point admits the feasible rank-one
lift~\eqref{eq:rank-one-lift}. Hence the optimal value of~\eqref{eq:sdp-dnn}
is a lower bound on that of the original binary quadratic problem.

For the dual derivation, the multiplier $L\in\R^{m\times(n+1)}$ for
$BY=0$ contributes $\ip{L}{BY}=\ip{\sym(B^TL)}{Y}$, where
$\sym(M):=(M+M^T)/2$. The adjoint of $\arrow$ is
\begin{equation}\label{eq:qcr-affine-adjoint}
  \arrow^*(z,u)
  =\begin{pmatrix}
    z&u^T/2\\
    u/2&-\Diag(u)
  \end{pmatrix}.
\end{equation}
The normalization and arrow equations, together with $Y\succeq0$,
imply nonnegativity of the diagonal and first-row and first-column entries.
Thus only $Y_{ij}\ge0$ for $1\le i<j\le n$ adds constraints.
We nevertheless retain nonnegativity of the entire lower-right block
and set the first-row and first-column multipliers to zero. With $N\in\Smat^n$ entrywise nonnegative, the multiplier is
$\begin{psmallmatrix}0&0\\0&N\end{psmallmatrix}$ and its Lagrangian
contribution is $-\ip{N}{Y_{1:n,1:n}}$.
For $z\in\R$ and $\lambda\in\R^n$, define
\begin{equation}\label{eq:dnn-base-slack}
  S(z,\lambda,N)
  :=C-\arrow^*(z,\lambda)-\begin{pmatrix}0&0\\0&N\end{pmatrix}
  =
  \begin{pmatrix}
      -z&(c-\lambda)^T/2\\
      (c-\lambda)/2&Q+\Diag(\lambda)-N
  \end{pmatrix}.
\end{equation}
The dual of~\eqref{eq:sdp-dnn} is therefore
\begin{equation}\label{eq:original-dual}
  \begin{aligned}
  \max_{z,\lambda,N,L}\quad &z\\
  \text{s.t.}\quad
      &S(z,\lambda,N)+\sym(B^TL)\succeq0,\\
      &N\ge0 \quad\text{entrywise}.
  \end{aligned}
\end{equation}
\subsection{The MIQCR reformulation}
\label{sec:dnn-miqcr}

The standard MIQCR construction \cite{billionnet2012extending} converts
a feasible point of the original dual into an exact convex reformulation
while retaining its dual lower bound. We present it
in our multiplier notation.

Fix a feasible point $(z,\lambda,N,L)$ of~\eqref{eq:original-dual}.
Partition $L=[\,\beta\ \ \Gamma\,]$, where
$\beta\in\R^m$ and $\Gamma\in\R^{m\times n}$.  The dual slack in
\eqref{eq:original-dual} is
\begin{equation}\label{eq:original-dnn-slack-block}
  \begin{pmatrix}
   -z-\beta^Tb
   &\tfrac12(c-\lambda-\Gamma^Tb+A^T\beta)^T\\[2pt]
   \tfrac12(c-\lambda-\Gamma^Tb+A^T\beta)
   &Q+\Diag(\lambda)-N+\sym(A^T\Gamma)
  \end{pmatrix}
  \succeq0.
\end{equation}

Define
\[
  q_{\lambda,N,L}(x)
  :=z+
  \begin{pmatrix}1\\x\end{pmatrix}^{\!T}
  \bigl(S(z,\lambda,N)+\sym(B^TL)\bigr)
  \begin{pmatrix}1\\x\end{pmatrix}.
\]
The positive semidefinite slack gives $q_{\lambda,N,L}(x)\ge z$ for every
$x\in\R^n$. Its lower-right principal block is the quadratic coefficient
matrix of $q_{\lambda,N,L}$, so this function is globally convex.
Expanding the block quadratic form cancels $z$ and gives
\begin{equation}\label{eq:dnn-underestimator}
  q_{\lambda,N,L}(x)=g(x)+\lambda^T(x\circ x-x)-x^TNx+(Ax-b)^T\Gamma x+\beta^T(Ax-b).
\end{equation}
At a feasible binary point, the binary and equality perturbations vanish,
leaving $q_{\lambda,N,L}(x)=g(x)-x^TNx\le g(x)$.
MIQCR restores the removed term $x^TNx$ through product variables.

For the binary problem~\eqref{eq:bqp}, let
$E\subseteq\{(i,j):1\le i<j\le n\}$ contain the off-diagonal support of $N$.
Introduce continuous variables $y_{ij}$ for $(i,j)\in E$ and impose the
McCormick system
\begin{equation}\label{eq:product-linking}
  0\le y_{ij},\qquad
  y_{ij}\le x_i,\qquad
  y_{ij}\le x_j,\qquad
  y_{ij}\ge x_i+x_j-1.
\end{equation}
For binary $x$, these constraints give $y_{ij}=x_ix_j$. Let $Y_E$ denote
the symmetric partial product matrix with $(Y_E)_{ii}=x_i$,
$(Y_E)_{ij}=(Y_E)_{ji}=y_{ij}$ for $(i,j)\in E$, and zero entries in the
remaining off-diagonal positions. The reformulated objective is
\begin{equation}\label{eq:miqcr-compact}
  q_{\lambda,N,L}(x)+\ip{N}{Y_E}.
\end{equation}
This objective is jointly convex because $q_{\lambda,N,L}$ is convex and
$Y_E$ enters linearly. At every feasible binary point, the product constraints
and the choice of $E$ give $\ip{N}{Y_E}=x^TNx$. The added linear term therefore
cancels precisely the term removed from $g(x)$, so the reformulated objective
equals $g(x)$. Thus every original-dual feasible point produces
a globally convex MIQP that is exact on the binary feasible set.
Since $q_{\lambda,N,L}(x)\ge z$ and $\ip{N}{Y_E}\ge0$, the continuous
reformulation retains the valid dual lower bound $z$.
\Cref{subsec:dual-feasible-repair} explains how to obtain such a tuple from
approximate multiplier estimates.

\subsection{The facially reduced relaxation and dual}
\label{sec:admm-implementation}
\label{sec:fr-dnn}

Facial reduction, introduced by Borwein and Wolkowicz
\cite{borwein1981facial}, addresses the failure of Slater's condition in
conic optimization by replacing the ambient cone with a smaller face
containing the feasible set. In semidefinite programming, this allows
feasible matrices to be represented in a lower-dimensional space; reaching
the minimal face restores strict feasibility relative to that face.
When $m>0$, the lifted equalities force the feasible matrices into a proper face
of the positive semidefinite cone. Exploiting this structure yields a smaller semidefinite representation
without changing the relaxation bound. The reduced formulation and its dual provide the setting for the recovery
and reformulation results in \cref{sec:dual-recovery}.

If~\eqref{eq:sdp-dnn} is feasible, its feasible set is compact: the diagonal
equations and positive semidefiniteness bound every entry of $Y$.  Its
dual~\eqref{eq:original-dual} is strictly feasible by setting
$N=\one\one^T$, $L=0$,
choosing $\lambda=\rho\one$ with $\rho$ sufficiently large, and then taking
$z<0$ with sufficiently large magnitude.  Hence there is no duality gap.
When $m>0$, however, $BY=0$ forces every feasible $Y$ to be
singular.  The original primal therefore fails Slater's condition, and
attainment of the original dual is not automatic.  Facial reduction makes
this degeneracy explicit.

The equations~\eqref{eq:BY0} place every feasible lifted matrix in a
face of $\Splus^{n+1}$, which is proper when $m>0$.  The matrix $B^TB\succeq0$ exposes this face:
for every feasible $Y\succeq0$,
\begin{equation}\label{eq:exposing-W}
  \ip{B^TB}{Y}=\tr(BYB^T)=0.
\end{equation}
Consequently $\range(Y)\subseteq\nullspace(B)$.  Let
$V\in\R^{(n+1)\times r}$ have orthonormal columns spanning
$\nullspace(B)$, where
\[
  r=n+1-m.
\]
Then every such matrix has the representation
\begin{equation}\label{eq:facial-param}
  Y=VRV^T,
  \qquad R\in\Splus^r.
\end{equation}
Using nonnegativity of the lower-right block of $Y$, the facially reduced
relaxation is
\begin{equation}\label{eq:reduced-dnn}
  \begin{aligned}
  \min_R\quad &\ip{V^TCV}{R}\\
  \text{s.t.}\quad
      &\arrow(VRV^T)=e_0,\\
      &(VRV^T)_{1:n,1:n}\ge0,\\
      &R\succeq0.
  \end{aligned}
\end{equation}
The parametrization~\eqref{eq:facial-param} is exact, so
\eqref{eq:reduced-dnn} and \eqref{eq:sdp-dnn} have the same optimal value.

The dual of~\eqref{eq:reduced-dnn} is
\begin{equation}\label{eq:reduced-dual}
  \begin{aligned}
  \max_{z,\lambda,N}\quad &z\\
  \text{s.t.}\quad
      &V^TS(z,\lambda,N)V\succeq0,\\
      &N\ge0 \quad\text{entrywise}.
  \end{aligned}
\end{equation}

The reduction removes the degeneracy caused by $BY=0$, but does not
necessarily restore strict feasibility of every remaining constraint.
Our recovery uses the full nullspace of $B$ and requires only reduced-dual
feasibility. If the reduced primal satisfies the appropriate Slater
condition, its finite optimal value is attained by the reduced dual.
The recovery below then transfers that optimal value to the original dual.
No additional facial reduction is needed for the finite dual repair in
\cref{subsec:dual-feasible-repair}.

\section{Dual recovery and convex reformulation}
\label{sec:dual-recovery}
Our goal is to construct exact convex MIQCR reformulations from multipliers
of the facially reduced relaxation. We first establish the dual-recovery
route to such a reformulation, then develop constructions for approximate
estimates, including one that enforces convexity directly.
Dual recovery is not guaranteed after facial reduction. For the SDP
relaxation considered here, the lifted equality system $BY=0$ provides
the multiplier corrections needed to recover an original-dual feasible
point from any feasible point of the reduced dual, without changing its objective
value. This system retains both the original equalities and their products
with the binary variables.

Approximate multiplier estimates need not form a feasible point of the reduced dual.
We therefore consider three ways to use such estimates:
restore dual feasibility by a spectral shift, optimize the dual bound
within a restricted multiplier family, or construct a convex reformulation
that optimizes the continuous MIQCR bound directly. The first two recover
dual certificates; the third uses the approximate multiplier estimates
directly to construct the Phase~2 objective. These are alternative
constructions, not successive steps.

The following lemma states the matrix-completion identity underlying the
proof of Hou, Tang, and Toh \cite[Theorem~4]{hou2026poly} independently of
the optimization model. For a given matrix $D$, a symmetric matrix $M$
can be made positive semidefinite by adding a term of the form
$\sym(D^T\Lambda)$ if and only if $M$ is positive semidefinite on
$\nullspace(D)$.

{
\begin{lemma}[Positive semidefinite completion]
\label[lemma]{lem:psd-completion}
Let $M\in\Smat^p$, let $D\in\R^{r\times p}$ have full row
rank, and let the columns of $U$ form an orthonormal basis of
$\nullspace(D)$. Then
\begin{equation}\label{eq:psd-completion-equivalence}
 U^TMU\succeq0
 \quad\Longleftrightarrow\quad
 \exists\,\Lambda\in\R^{r\times p}:\ M+\sym(D^T\Lambda)\succeq0.
\end{equation}
In particular, with $P=UU^T$, the choice
\begin{equation}\label{eq:psd-completion-multiplier}
 \Lambda=-(DD^T)^{-1}DM(I+P)
\end{equation}
satisfies
\begin{equation}\label{eq:psd-completion-identity}
 M+\sym(D^T\Lambda)=PMP.
\end{equation}
\end{lemma}
\begin{proof}
Since $P$ is the orthogonal projector onto $\nullspace(D)$, its
complementary projector is $I-P=D^T(DD^T)^{-1}D$. The choice of
$\Lambda$ in~\eqref{eq:psd-completion-multiplier} therefore satisfies
\[
 \sym(D^T\Lambda)
 =-\sym\bigl((I-P)M(I+P)\bigr)=PMP-M.
\]
This proves~\eqref{eq:psd-completion-identity}. If $U^TMU\succeq0$, it follows that
\[
 M+\sym(D^T\Lambda)=PMP=U(U^TMU)U^T\succeq0,
\]
proving the forward implication in~\eqref{eq:psd-completion-equivalence}.

Conversely, if $M+\sym(D^T\Lambda)\succeq0$ for some $\Lambda$,
then $DU=0$ gives
\[
 U^TMU
 =U^T\bigl(M+\sym(D^T\Lambda)\bigr)U\succeq0.
\]
\end{proof}
}

In the applications below, $D$ represents the equality constraints and
$\Lambda$ their multipliers. We use the identity both to recover the
lifted-equality multiplier and to construct the Phase~2 quadratic objective.

\subsection{Dual recovery from a feasible point of the reduced dual}
\label{subsec:exact-dual-recovery}
Given a feasible point $(z,\lambda,N)$ of the reduced
dual~\eqref{eq:reduced-dual}, the task is to recover an equality multiplier $L$
that makes $(z,\lambda,N,L)$ feasible for the original
dual~\eqref{eq:original-dual}, without changing the dual value. This recovered
solution supplies the MIQCR reformulation in \cref{sec:dnn-miqcr}.

The condition $V^TSV\succeq0$ ensures that $S$ is positive semidefinite
on $\ker B$, but does not ensure that $S$ is positive semidefinite on the
whole space. The multipliers of the lifted equality system $BY=0$ generate every
symmetric correction $T$ satisfying $V^TTV=0$.
Indeed, for $F:=\{Y\succeq0:BY=0\}$, with orthogonal complement taken
in $\Smat^{n+1}$,
\[
 F^\perp=\{T:V^TTV=0\}
 =\{\sym(B^TL):L\in\R^{m\times(n+1)}\}.
\]
These identities explain the geometric basis of \cref{lem:psd-completion}
in the present setting. The terms $\sym(B^TL)$ span $F^\perp$, and the
decomposition $F^*=\Splus^{n+1}+F^\perp$, an instance of the niceness
property studied by Pataki \cite{pataki2012nice}, guarantees that a slack
positive semidefinite on $\ker B$ can be made positive semidefinite by
such a correction. \Cref{lem:psd-completion} provides an explicit multiplier
realizing this correction.

With $B$ and $V$ as defined above, set
\[
  \Pi:=VV^T=I-B^T(BB^T)^{-1}B.
\]
\begin{corollary}[Recovery of the equality multiplier]
\label[corollary]{prop:dual-recovery}
For $S\in\Smat^{n+1}$,
\[
 V^TSV\succeq0
 \quad\Longleftrightarrow\quad
 \exists\,L\in\R^{m\times(n+1)}:\ S+\sym(B^TL)\succeq0.
\]
\end{corollary}
\begin{proof}
Apply \cref{lem:psd-completion} with $M=S$, $D=B$, $U=V$, and
$\Lambda=L$. Its projector is $P=\Pi$. In particular, if $V^TSV\succeq0$,
the choice
\begin{equation}\label{eq:direct-recovery-L}
  L:=-(BB^T)^{-1}BS(I+\Pi)
\end{equation}
satisfies
\begin{equation}\label{eq:direct-recovered-slack}
  S+\sym(B^TL)=\Pi S\Pi=V(V^TSV)V^T\succeq0.
\end{equation}
The converse is the necessity direction of the lemma.
\end{proof}
Apply the corollary to $S=S(z,\lambda,N)$ for a feasible point of
\eqref{eq:reduced-dual}. Recovery changes only $L$, so the original-dual
point has the same $z$, $\lambda$, and $N$. In particular, the dual value
is preserved. When the reduced primal satisfies Slater's condition, an
optimal solution of the reduced dual therefore gives an optimal solution
of the original dual whose objective value equals the optimal value
of~\eqref{eq:sdp-dnn}.

\subsection{Dual recovery from approximate multipliers: spectral shift}
\label{subsec:admm-postprocessing}
\label{subsec:dual-feasible-repair}

The following proposition starts from arbitrary estimates of the normalization,
arrow-equation, and nonnegativity multipliers. It enforces nonnegativity of
$N$ and restores reduced-dual feasibility by a spectral shift.
\begin{proposition}[Finite dual repair]\label[proposition]{prop:dnn-dual-repair}
Let $\widetilde z\in\R$, $\widetilde\lambda\in\R^n$, and
$\widetilde N\in\Smat^n$ be arbitrary multiplier estimates. Set
\[
 N:=\max\{\widetilde N,0\}\quad\text{entrywise},
 \qquad
 \widetilde S:=S(\widetilde z,\widetilde\lambda,N).
\]
Define
\begin{equation}\label{eq:dnn-repair-direction}
 T_0:=
 \begin{pmatrix}
   n&-\one^T/2\\
   -\one/2&I
 \end{pmatrix},
 \qquad
 V^TT_0V=R_T^TR_T,
\end{equation}
where $R_T$ is the Cholesky factor. Let
\begin{equation}\label{eq:dnn-repair-shift}
 \delta_*
 :=\max\left\{0,-\lambda_{\min}\!\left(
 R_T^{-T}V^T\widetilde S V R_T^{-1}
 \right)\right\}.
\end{equation}
For any $\delta\ge\delta_*$, set
\begin{equation}\label{eq:dnn-repaired-multipliers}
 z_\delta:=\widetilde z-n\delta,
 \qquad
 \lambda_\delta:=\widetilde\lambda+\delta\one,
 \qquad
 S_\delta:=S(z_\delta,\lambda_\delta,N).
\end{equation}
Then $(z_\delta,\lambda_\delta,N)$ is feasible for the reduced
dual~\eqref{eq:reduced-dual}. Applying \cref{prop:dual-recovery} yields
an original-dual feasible point with the same value $z_\delta$.
\end{proposition}
\begin{proof}
Entrywise clipping preserves symmetry and gives $N\ge0$.
The Schur complement of the identity block in $T_0$ is $3n/4>0$.
Thus $T_0\succ0$ and, since $V$ has full column rank, $V^TT_0V\succ0$,
so $R_T$ is invertible. The multiplier changes give
\[
 S_\delta-\widetilde S
 =-\arrow^*(-n\delta,\delta\one)=\delta T_0.
\]
Consequently,
\[
 R_T^{-T}(V^TS_\delta V)R_T^{-1}
 =R_T^{-T}(V^T\widetilde S V)R_T^{-1}+\delta I\succeq0,
\]
where the inequality follows from $\delta\ge\delta_*$.
Congruence therefore gives $V^TS_\delta V\succeq0$.
Together with $N\ge0$, this proves reduced-dual feasibility.
The original-dual extension follows
from \cref{prop:dual-recovery}.
\end{proof}
The repaired value $z_\delta=\widetilde z-n\delta$ is a valid dual lower
bound, and the MIQCR construction in \cref{sec:dnn-miqcr} gives an exact
convex reformulation whose continuous relaxation has optimal value at least
$z_\delta$.
The correction guarantees feasibility, but its cost $n\delta$ need not
be small; the initial estimate $\widetilde z$ need not itself be valid.
\subsection{Recovery from approximate multipliers: optimizing the dual bound}
\label{sec:optimal-facial-repair}
The spectral-shift repair changes the normalization and arrow-equation
multipliers through a single parameter $\delta$. We now optimize the
normalization and a scalar adjustment to the arrow multipliers jointly.
We show that this dual-bound optimization
is equivalent to a quadratic minimization on a sphere, allowing efficient
computation of the best dual certificate within the chosen multiplier family.

Let $\widetilde\lambda\in\R^n$ and $\widetilde N\in\Smat^n$ be given
multiplier estimates. As in the spectral-shift construction, set
$N=\max\{\widetilde N,0\}$ entrywise and define the centered arrow multiplier
$\bar\lambda=\widetilde\lambda-(\one^T\widetilde\lambda/n)\one$.
We keep $N$ and $\bar\lambda$ fixed and restrict the arrow multiplier to the family
\begin{equation}\label{eq:lambda-decomposition}
 \lambda=\bar\lambda+\mu\one,\qquad \mu\in\R,
\end{equation}
and optimize only $(z,\mu)$, subject to reduced-dual feasibility.

The term $\mu\one$ in~\eqref{eq:lambda-decomposition} contributes $\mu\sum_i(x_i^2-x_i)$
to the quadratic form. Every binary point satisfies $\sum_i(x_i^2-x_i)=0$,
which is a sphere equation by the identity
\begin{equation}\label{eq:ordinary-binary-sphere}
 \sum_i(x_i^2-x_i)
 =\left\|x-\tfrac12\one\right\|^2-\frac n4.
\end{equation}
This identity places every binary point on a sphere, but does not yet
account for the equalities $Ax=b$. To incorporate them, we parameterize
the feasible affine space and express its intersection with the sphere
in reduced coordinates.
\label{subsec:binary-sphere}
\label{subsec:facial-loading-threshold}
Assume the binary feasible set is nonempty. Let $d=n-m$ and let
$W\in\R^{n\times d}$ have orthonormal columns spanning $\nullspace(A)$.
Project the sphere center $\tfrac12\one$ onto the affine space $Ax=b$
and define the squared radius of the intersection:
\begin{equation}\label{eq:binary-sphere-center}
 x_c=\frac12\one+A^T(AA^T)^{-1}
             \left(b-\frac12A\one\right),
 \qquad
 \rho=\frac n4-\left\|x_c-\frac12\one\right\|^2.
\end{equation}
Every point satisfying $Ax=b$ has the unique representation
\begin{equation}\label{eq:affine-null-param}
 x=x_c+Wu,\qquad
 \widetilde V:=\begin{pmatrix}1&0\\x_c&W\end{pmatrix},\qquad
 \begin{pmatrix}1\\x\end{pmatrix}
 =\widetilde V\begin{pmatrix}1\\u\end{pmatrix}.
\end{equation}
Substituting $x=x_c+Wu$ into the sphere identity
\eqref{eq:ordinary-binary-sphere}, using $W^TW=I$ and the orthogonality of
$x_c-\tfrac12\one$ to $\nullspace(A)$, gives
\begin{equation}\label{eq:binary-sphere-identity}
 \sum_i(x_i^2-x_i)=\|u\|^2-\rho.
\end{equation}
If $d=0$ or $\rho=0$, the nonempty binary feasible set contains only
$x_c$ and can be handled directly. We therefore assume $d\ge1$ and
$\rho>0$ in the following construction.

Since the columns of $V$ and $\widetilde V$ span the same nullspace,
$V^TSV\succeq0$ is equivalent to $\widetilde V^TS\widetilde V\succeq0$.
Set
\begin{equation}\label{eq:K-bar-lambda-N}
 K:=\widetilde V^TS(0,\bar\lambda,N)\widetilde V
   =\begin{pmatrix}a&g^T\\g&H\end{pmatrix}.
\end{equation}
In the original coordinates, changing the multipliers from
$(0,\bar\lambda)$ to $(z,\bar\lambda+\mu\one)$ changes the slack by
\[
 \Delta S:=S(z,\bar\lambda+\mu\one,N)-S(0,\bar\lambda,N)
 =\begin{pmatrix}
 -z&-\mu\one^T/2\\
 -\mu\one/2&\mu I_n
 \end{pmatrix}.
\]
Transforming this change to the reduced coordinates gives
\[
 \begin{aligned}
 \widetilde V^T\Delta S\widetilde V
 &=\begin{pmatrix}
 -z+\mu(\|x_c\|^2-\one^Tx_c)&\mu(x_c-\tfrac12\one)^TW\\
 \mu W^T(x_c-\tfrac12\one)&\mu W^TW
 \end{pmatrix}\\
 &=\begin{pmatrix}-z-\rho\mu&0\\0&\mu I_d\end{pmatrix},
 \end{aligned}
\]
where $W^TW=I_d$, $W^T(x_c-\tfrac12\one)=0$, and
$\|x_c\|^2-\one^Tx_c=-\rho$ by~\eqref{eq:binary-sphere-center}.
Thus the multiplier change is block diagonal in the reduced coordinates.
Adding it to $K$ yields the reduced slack
\begin{equation}\label{eq:common-repair-compression}
 \widetilde V^TS(z,\bar\lambda+\mu\one,N)\widetilde V
 =\begin{pmatrix}a-z-\rho\mu&g^T\\g&H+\mu I\end{pmatrix}.
\end{equation}
Dual-bound optimization maximizes the dual objective $z$
over $(z,\mu)$, subject to positive semidefiniteness of the reduced slack
in~\eqref{eq:common-repair-compression}.
The following standard equivalence expresses this optimization as a
quadratic minimization on the sphere.

\begin{theorem}[Trust-region characterization of the optimized repair]
\label{thm:optimal-facial-tr-repair}
For $d\ge1$, $\rho>0$, and
$K=\begin{psmallmatrix}a&g^T\\g&H\end{psmallmatrix}\in\Smat^{d+1}$,
the two problems
\begin{align}
 &\max_{z,\mu}\left\{z:
 \begin{pmatrix}a-z-\rho\mu&g^T\\g&H+\mu I\end{pmatrix}\succeq0
 \right\},\label{eq:repair-two-variable-sdp}\\
 &\min_{\|u\|^2=\rho}\{a+2g^Tu+u^THu\}
 \label{eq:repair-trust-region}
\end{align}
have equal, attained optimal values.
\end{theorem}

\begin{proof}
The semidefinite dual of \eqref{eq:repair-two-variable-sdp} is
\begin{equation}\label{eq:repair-two-constraint-sdp}
 \min_{R\succeq0}\{\ip{K}{R}:R_{00}=1,\
             \tr(R_{1:d,1:d})=\rho\}.
\end{equation}
Its feasible set is compact because $\tr R=1+\rho$, and
$\Diag(1,\rho/d,\ldots,\rho/d)$ is strictly feasible.
The maximization problem is also strictly feasible: choose
$\mu>-\lambda_{\min}(H)$ and then $z$ sufficiently negative.
Hence both optima are attained and their values agree.
An optimal extreme point of \eqref{eq:repair-two-constraint-sdp} has rank
one. Indeed, Pataki's rank bound \cite[Theorem~2.1]{pataki1998rank}
gives $r(r+1)/2\le2$ for a rank-$r$ extreme point under two affine constraints.
Since $R_{00}=1$, the rank is exactly one. Write
$R=[1;u][1;u]^T$; the trace constraint becomes $\|u\|^2=\rho$,
and its objective is the quadratic in \eqref{eq:repair-trust-region}.
\end{proof}

Problem~\eqref{eq:repair-trust-region} is an equality-constrained
trust-region subproblem that can be solved efficiently
\cite[Section~2]{hager2001sphere}. The corresponding optimal pair
$(z,\mu)$ in~\eqref{eq:repair-two-variable-sdp} gives a feasible
point $(z,\bar\lambda+\mu\one,N)$ of the reduced dual.
Applying \cref{prop:dual-recovery} then recovers the equality
multiplier $L$, yielding an original-dual feasible point with
the same objective value $z$.

The spectral-shift repair in~\eqref{eq:dnn-repaired-multipliers} is feasible
for~\eqref{eq:repair-two-variable-sdp}, so dual-bound optimization cannot
decrease its dual value in exact arithmetic. This guarantee concerns the certificate; it does not imply a
stronger continuous MIQCR bound or faster integer optimization. The optimization
still fixes $\bar\lambda$ and $N$, so it need not attain the DNN optimum.

\subsection{Reformulation from approximate multipliers: optimizing the continuous MIQCR bound}
\label{subsec:root-oriented-reformulation}
Dual recovery provides a sufficient condition for an exact convex MIQCR
reformulation by making the full homogeneous dual slack positive
semidefinite. Convexity, however, requires only a positive semidefinite
quadratic coefficient matrix. We now apply the same completion formula
directly to this matrix and choose the shift to maximize the continuous
MIQCR bound. This construction need not produce a feasible point of
\eqref{eq:reduced-dual}; its purpose is an exact convex reformulation,
rather than a dual certificate.

We use the same centered multiplier $\bar\lambda$ and
entrywise nonnegative matrix $N$ defined in \cref{sec:optimal-facial-repair},
keep both fixed, and optimize only $\mu$
in~\eqref{eq:lambda-decomposition}. To see why, recall
\eqref{eq:dnn-underestimator}: the normalization multiplier $z$
cancels when the quadratic objective is expanded, while the equality
multipliers contribute
\[
 (Ax-b)^T\Gamma x+\beta^T(Ax-b).
\]
These terms vanish at every point satisfying $Ax=b$, including
fractional points in the continuous relaxation. The equality
multipliers can therefore change the curvature outside this affine
space without changing the objective on the feasible set.
With $\bar\lambda$, $N$, and the product variables and McCormick constraints fixed,
only $\mu$ changes the continuous relaxation objective, through
the term $\mu\sum_i(x_i^2-x_i)$.

Define
\begin{equation}\label{eq:root-centered-block}
 K_0:=Q+\Diag(\bar\lambda)-N.
\end{equation}
Recall that $d=n-m$. If $d=0$, the equalities determine a unique candidate $x$;
checking whether it is binary solves the original BQP. We therefore assume
$d>0$. Using the nullspace basis $W$ from the preceding subsection, set
\begin{equation}\label{eq:root-reduced-hessian}
 H_0:=W^TK_0W.
\end{equation}
Nohra, Raghunathan, and Sahinidis \cite[Sections~3.3--3.4]{nohra2021spectral}
study nullspace eigenvalue shifts and show that the smallest admissible
nonnegative shift gives the strongest relaxation in their scalar-shift
family. Here we apply this principle to $K_0$, with the MIQCR multipliers
and product constraints fixed. Binary-point exactness allows $\mu\in\R$,
and the explicit equality completion below makes the resulting quadratic
objective globally convex.

For a given $\mu$, consider the quadratic function
\[
 g(x)+(\bar\lambda+\mu\one)^T(x\circ x-x)-x^TNx,
\]
whose quadratic coefficient matrix is $K_0+\mu I$.
This matrix need not be positive semidefinite.
Adding the equality term $(Ax-b)^T\Gamma x$ changes the quadratic
coefficient matrix to $K_0+\mu I+\sym(A^T\Gamma)$.
We therefore seek the smallest shift $\mu$ for which a choice of
$\Gamma$ makes the resulting quadratic function convex.

\begin{proposition}[Smallest convexifying parameter]
\label[proposition]{prop:root-minimum-curvature}
The smallest $\mu\in\R$ for which there exists
$\Gamma\in\R^{m\times n}$ satisfying
\[
 K_0+\mu I+\sym(A^T\Gamma)\succeq0
\]
is
\begin{equation}\label{eq:root-mu-star}
 \mu_{\star}=-\lambda_{\min}(H_0).
\end{equation}
\end{proposition}

\begin{proof}
By \cref{lem:psd-completion} with $M=K_0+\mu I$, $D=A$,
$U=W$, and $\Lambda=\Gamma$,
\[
 \exists\,\Gamma\in\R^{m\times n}:\ K_0+\mu I+\sym(A^T\Gamma)\succeq0
 \quad\Longleftrightarrow\quad
 H_0+\mu I\succeq0,
\]
since $W^T(K_0+\mu I)W=H_0+\mu I$. The latter condition holds exactly
when $\mu\ge-\lambda_{\min}(H_0)$.
\end{proof}

To construct the reformulation, set $K_{\star}:=K_0+\mu_{\star}I$
and let
\[
 P:=WW^T=I-A^T(AA^T)^{-1}A.
\]
The multiplier formula in \cref{lem:psd-completion}, with
$M=K_{\star}$, $D=A$, $U=W$, and $\Lambda=\Gamma_{\star}$, gives
\begin{equation}\label{eq:root-Gamma}
 \Gamma_{\star}
 :=-(AA^T)^{-1}AK_{\star}(I+P),
\end{equation}
for which
\begin{equation}\label{eq:root-completion}
 K_{\star}+\sym(A^T\Gamma_{\star})
 =PK_{\star}P
 =W(H_0+\mu_{\star}I)W^T\succeq0.
\end{equation}

With the shift $\mu_{\star}$ and equality perturbation
$(Ax-b)^T\Gamma_{\star}x$, define the quadratic function
\begin{equation}\label{eq:root-underestimator}
 \begin{aligned}
 q_{\star}(x)
 ={}&g(x)
 +(\bar\lambda+\mu_{\star}\one)^T(x\circ x-x)
 -x^TNx\\
 &+(Ax-b)^T\Gamma_{\star}x.
 \end{aligned}
\end{equation}
By \eqref{eq:root-completion}, $q_{\star}$ is globally convex.
At every feasible binary point, the binary and equality perturbations
vanish, giving $q_{\star}(x)=g(x)-x^TNx$. Restoring $x^TNx$
through the product variables and McCormick constraints in
\cref{sec:dnn-miqcr} therefore yields an exact convex MIQP.
The multiplier $N$, product variables, and McCormick constraints remain fixed
throughout this construction.

We now show that $\mu_{\star}$ maximizes the continuous MIQCR bound
within the restricted multiplier family.
For fixed $(\bar\lambda,N)$, changing $\mu$ from $\mu_1$ to
$\mu_2$ changes the Phase~2 objective on $Ax=b$ by
\begin{equation}\label{eq:common-shift-phase2-difference}
 (\mu_2-\mu_1)\sum_{i=1}^n(x_i^2-x_i).
\end{equation}
For $0\le x\le1$,
\[
 \sum_i(x_i^2-x_i)=-\sum_i x_i(1-x_i)\le0.
\]
Thus increasing $\mu$ weakens the pointwise objective of the
continuous minimization problem.

\begin{corollary}[Optimality of the minimum convexifying shift]
\label[corollary]{thm:root-dominance}
Fix $(\bar\lambda,N)$ and use the same product variables and McCormick constraints. Let
$r(\mu)$ denote the optimal value of the continuous MIQCR relaxation formed
with parameter $\mu$ in~\eqref{eq:lambda-decomposition} and any equality completion that makes its quadratic
coefficient positive semidefinite.  Then every such $\mu$ satisfies
$\mu\ge\mu_{\star}$ and
\begin{equation}\label{eq:root-monotonicity}
 \mu_{\star}\le\mu_1\le\mu_2
 \quad\Longrightarrow\quad
 r(\mu_{\star})\ge r(\mu_1)\ge r(\mu_2).
\end{equation}
When the three constructions are applied to the same multiplier
estimates and use the same product variables and McCormick constraints, the choice
$\mu=\mu_{\star}$ gives a continuous MIQCR bound no weaker than
those obtained by spectral-shift recovery or dual-bound optimization.
\end{corollary}

\begin{proof}
By \cref{prop:root-minimum-curvature}, every admissible shift satisfies
$\mu\ge\mu_{\star}$. Equation~\eqref{eq:common-shift-phase2-difference}
shows that increasing $\mu$ decreases the objective pointwise on the
common continuous feasible set, since $0\le x\le1$. Taking minima
proves \eqref{eq:root-monotonicity}. Both recovery procedures yield
positive-semidefinite quadratic coefficients, so their shifts are
admissible and the comparison follows.
\end{proof}

The three constructions offer different ways to turn approximate
multiplier estimates into an exact convex MIQCR reformulation.
Spectral-shift recovery restores dual feasibility explicitly, while
dual-bound optimization chooses the best dual certificate within the
restricted multiplier family. Continuous-bound optimization instead
chooses the smallest convexifying shift, which gives the strongest
continuous MIQCR bound within that family, using fixed product variables and McCormick
constraints. The latter two objectives need not select the same shift.

\Cref{tab:multiplier-constructions} summarizes the information used and the
multipliers constructed in each approach.

\begin{table}[ht]
\centering\small
\caption{Inputs and constructed multipliers for the three reformulation methods.}
\label{tab:multiplier-constructions}
\begin{tabularx}{\textwidth}{@{}>{\raggedright\arraybackslash}p{0.23\textwidth}>{\raggedright\arraybackslash}p{0.21\textwidth}>{\raggedright\arraybackslash}p{0.25\textwidth}>{\raggedright\arraybackslash}X@{}}
\toprule
Construction & Starting estimates used & Parameters chosen & Equality multipliers derived\\
\midrule
\ref{subsec:dual-feasible-repair} Spectral shift &
$\widetilde z,\widetilde\lambda,\widetilde N$ &
$\delta$, determining $z,\lambda$ & $L=[\,\beta\ \ \Gamma\,]$\\[5pt]
\ref{sec:optimal-facial-repair} Dual-bound optimization &
$\widetilde\lambda,\widetilde N$ &
$z,\mu$, with $\lambda=\bar\lambda+\mu\one$ & $L=[\,\beta\ \ \Gamma\,]$\\[5pt]
\ref{subsec:root-oriented-reformulation} Continuous-bound optimization &
$\widetilde\lambda,\widetilde N$ &
$\mu$, with $\lambda=\bar\lambda+\mu\one$ & $\Gamma$ only\\
\bottomrule
\end{tabularx}
\end{table}

In all three constructions, $N$ and $\bar\lambda$ remain fixed.
The first two recover the full equality multiplier
$L=[\,\beta\ \ \Gamma\,]$ to obtain a feasible dual certificate;
the third constructs only $\Gamma$ to ensure convexity.
We use this direct continuous-bound construction in the computational
experiments that follow.

\section{Computational experiments}
\label{sec:computational-design}

\subsection{The two-phase workflow}
\label{sec:pipeline}

We implement MIQCR--FR using ADMM to solve the facially reduced
relaxation. The two-phase workflow is as follows. Phase~1 uses
continuous-bound optimization to construct an exact convex reformulation
from the approximate ADMM multipliers. Phase~2 solves that reformulation.
The bound guarantee in \cref{thm:root-dominance} motivates this
Phase~1 choice: from the same multiplier estimates and using the same
product variables and McCormick constraints, continuous-bound optimization
yields a continuous
relaxation at least as strong as either dual-recovery construction.
Preliminary computational tests also favored it over these alternatives
in overall solution performance.

\paragraph{Phase 1: Convex reformulation.}
Phase~1 includes the relaxation solve, multiplier extraction, and model
construction.
For the equality-constrained BQP~\eqref{eq:bqp}, these steps are:
\begin{enumerate}[leftmargin=*]
 \item \emph{Solve the reduced relaxation.}
 Remove redundant rows of $A$, compute an orthonormal basis $V$ of
 $\nullspace([-b\ A])$, and apply ADMM to the split formulation
 \eqref{eq:admm-split}.
 \item \emph{Extract the multiplier estimates.}
 Obtain $(\widetilde\lambda,\widetilde N)$ from the finite ADMM output using
 \eqref{eq:multiplier-estimates}. Set $N=\widetilde N\ge0$ and
 $\bar\lambda=\widetilde\lambda-(\one^T\widetilde\lambda/n)\one$.
 \item \emph{Build the convex MIQP.}
 With $\bar\lambda$ and $N$ fixed, choose $\mu_{\star}$ and
 $\Gamma_{\star}$ as in \cref{subsec:root-oriented-reformulation}.
 Form the quadratic objective in~\eqref{eq:root-underestimator} and restore
 the product terms using the MIQCR construction in \cref{sec:dnn-miqcr}.
\end{enumerate}

\paragraph{Phase 2: Integer optimization.}
Solve the resulting convex MIQP with a mixed-integer solver, using its
continuous relaxation for bounding. The feasible objectives and valid
lower bounds refer to the original BQP because the reformulation is exact
at every feasible binary point. Its final optimality gap, together with
the total time for both phases, measures the performance of the complete
method.

\subsection{ADMM implementation and multiplier extraction}\label{subsec:admm-implementation-details}

The ADMM solver used in our implementation provides the multiplier
estimates required by the constructions in \cref{sec:dual-recovery}.

\paragraph{Split formulation and dual.}

To apply two-block ADMM with separate positive-semidefinite and polyhedral
projections, write~\eqref{eq:reduced-dnn} as
\begin{equation}\label{eq:admm-split}
  \min_{R,Y}\ \ip{C}{Y}
  \quad\text{s.t.}\quad
  Y=VRV^T,\quad R\succeq0,\quad \arrow(Y)=e_0,\quad Y\ge0.
\end{equation}
With multiplier $Z$ for $Y-VRV^T=0$, a direct
Lagrange calculation gives the split dual
\begin{equation}\label{eq:split-dual}
  \begin{aligned}
  \max_{z,\lambda,N,Z}\quad &z\\
  \text{s.t.}\quad
     &S(z,\lambda,N)+Z=0,\\
     &V^TZV\preceq0,\\
      &N\ge0 \quad\text{entrywise}.
  \end{aligned}
\end{equation}
Together with $V^TZV\preceq0$, the stationarity equation gives
\begin{equation}\label{eq:split-reduced-slack}
  -V^TZV
  =V^TS(z,\lambda,N)V\succeq0.
\end{equation}
Eliminating $Z$ recovers~\eqref{eq:reduced-dual}.  Thus the split and reduced
duals are algebraically equivalent; no ADMM convergence argument is involved.

\paragraph{ADMM iterations.}

We apply standard two-block ADMM to~\eqref{eq:admm-split}. Define
$\cP:=\{Y\in\Smat^{n+1}:\arrow(Y)=e_0,\ Y\ge0\}$. For penalty
parameter $\sigma>0$, relaxation parameter $\omega$, and unscaled symmetric
multiplier $Z\in\Smat^{n+1}$, the iterations are
\begin{subequations}\label{eq:admm-updates}
\begin{align}
  R^{k+1}
  &=\Pi_{\Splus}
    \left(V^T\left(Y^k+\frac{Z^k}{\sigma}\right)V\right),\label{eq:R-update}\\
  Y^{k+1}
  &=\Pi_{\cP}
    \left(VR^{k+1}V^T-\frac{C+Z^k}{\sigma}\right),\label{eq:Y-update}\\
  Z^{k+1}
  &=Z^k+\omega\sigma\left(Y^{k+1}-VR^{k+1}V^T\right).
  \label{eq:Z-update}
\end{align}
\end{subequations}
For $\bar Y\in\Smat^{n+1}$, the projection
$\widehat Y:=\Pi_{\cP}(\bar Y)$ is
\begin{equation}\label{eq:dnn-polyhedral-projection}
\begin{aligned}
  \widehat Y_{00}&=1,\\
  \widehat Y_{0i}=\widehat Y_{i0}=\widehat Y_{ii}
    &=\max\left\{0,\frac{2\bar Y_{0i}+\bar Y_{ii}}{3}\right\},
      &&i=1,\ldots,n,\\
  \widehat Y_{ij}=\widehat Y_{ji}
    &=\max\{0,\bar Y_{ij}\},
      &&1\le i<j\le n.
\end{aligned}
\end{equation}
Thus each iteration uses one reduced positive-semidefinite projection and one
explicit polyhedral projection. For QAP, we adapt the implementation to the strengthened relaxation
of Oliveira, Wolkowicz, and Xu~\cite{oliveira2018admm}.

\paragraph{Multiplier extraction.}

Given the final ADMM coupling multiplier $\widetilde Z$, we extract estimates
$(\widetilde z,\widetilde\lambda,\widetilde N)$ from an approximate form
of the stationarity equation $S(z,\lambda,N)+Z=0$
in~\eqref{eq:split-dual}, namely
\[
 C+\widetilde Z\approx
 \arrow^*(\widetilde z,\widetilde\lambda)
 +\begin{pmatrix}0&0\\0&\widetilde N\end{pmatrix},
 \qquad \widetilde N\ge0.
\]
We extract the multiplier estimates by minimizing the Frobenius norm of
the stationarity residual:
\begin{equation}\label{eq:multiplier-estimates}
 (\widetilde z,\widetilde\lambda,\widetilde N)
 \in\operatorname*{arg\,min}_{\substack{z\in\R,\ \lambda\in\R^n\\
                                    N\in\Smat^n,\ N\ge0}}
 \left\|C+\widetilde Z-\arrow^*(z,\lambda)
 -\begin{pmatrix}0&0\\0&N\end{pmatrix}\right\|_F^2.
\end{equation}
Problem~\eqref{eq:multiplier-estimates} has a closed-form solution.
Write $T=C+\widetilde Z$. The normalization and off-diagonal entries give
\[
 \widetilde z=T_{00},\qquad
 \widetilde N_{ij}=\max\{T_{ij},0\}
 \quad(1\le i\ne j\le n).
\]
For each $i=1,\ldots,n$, the remaining terms reduce to
\[
 \min_{\substack{\lambda_i\in\R\\N_{ii}\ge0}}
 2\left(T_{0i}-\frac{\lambda_i}{2}\right)^2
 +\left(T_{ii}+\lambda_i-N_{ii}\right)^2,
\]
where $\lambda_i$ is unrestricted and the factor $2$ accounts for the two
symmetric off-diagonal entries. Minimizing this expression gives
\[
 \widetilde\lambda_i=
 \begin{cases}
  2T_{0i},&T_{ii}+2T_{0i}\ge0,\\[2pt]
  \dfrac23(T_{0i}-T_{ii}),&T_{ii}+2T_{0i}<0,
 \end{cases}
 \qquad
 \widetilde N_{ii}=\max\{T_{ii}+2T_{0i},0\}.
\]
Thus multiplier extraction requires $O(n^2)$ arithmetic operations and no
iterative optimization solver.
This enforces $\widetilde N\ge0$ but does not guarantee
$V^TS(\widetilde z,\widetilde\lambda,\widetilde N)V\succeq0$.
The estimates can be used in any of the three constructions in
\cref{sec:dual-recovery}.

\subsection{Instances and experimental protocol}
\label{subsec:experimental-protocol}

In the results below, MIQCR--FR denotes the ADMM implementation described
above. We compare MIQCR--FR, MIQCR-CB, and direct Gurobi on 431 instances
from six problem families, using the same inputs and initial feasible
solutions for all three methods. Their relaxations and reformulations differ.

The benchmark includes 61 symmetric-input QAPLIB instances of order at
most 30 \cite{burkard1997qaplib}; 102 quadratic cycle-cover instances with
at most 1,000 arcs from de Meijer and Sotirov's collection
\cite{demeijer2021qccp}; all 30 QSAP1 semi-assignment instances distributed
by The Optimization Firm \cite{nohra2021spectral}; and 90 $k$-cluster
instances from Lambert's repository \cite{lambertKclusterData}, covering
$n=80,120,160$ and $k=n/4,n/2$ with all available densities and replicates.
We also construct 72 prescribed-size partitioning instances from 36 graphs
in the DIMACS/Stanford GraphBase collection hosted by Trick
\cite{trickColorData}, using three or four nearly equal parts per graph.
These are derived instances, rather than an existing quadratic partitioning
benchmark. The remaining 76 instances are quadratic shortest-path (QSPP)
problems from Hu and Sotirov \cite{hu2020qspp}: 40 GRID1 square, eight
GRID3 square, and 28 PARK cases, with 60--990 binary arc variables.
The instance manifest fixes the selection and ordering; Tables S1--S6 of
Online Resource~1 report every instance.

Each run uses one AMD EPYC 9654 core and 64 GiB of memory, MATLAB R2025b,
and Gurobi 13.0.2. MATLAB, its underlying linear algebra routines, and
Gurobi each use one thread. Each method receives a total budget of
3,600 seconds. For MIQCR--FR and MIQCR-CB, the Phase~1 solver is capped
at 900 seconds, and Phase~2 receives the time remaining after
reformulation. Direct Gurobi receives the full budget. We report the
complete reformulation time (P1), subsequent integer optimization time
(P2), and their sum (Total).
Reported times may slightly exceed the nominal budget because of
noninterruptible numerical work and solver termination.

ADMM uses residual tolerance $10^{-3}$, at most 10,000 iterations, and at
most eight penalty changes. For MIQCR--FR, we apply a numerical
convexity safeguard to the quadratic coefficient matrix $Q_c$ after
equality completion: we add $\varepsilon I$ to $Q_c$ and subtract
$\varepsilon\one$ from the linear coefficient, where
$\varepsilon=\max\{0,10^{-8}\max\{1,\|Q_c\|_\infty\}-\lambda_{\min}(Q_c)\}$.
This preserves the objective at binary points and weakens the continuous
bound by at most $n\varepsilon/4$.

MIQCR-CB uses the SMIQP implementation, modified to save completed
reformulations during Phase~1; the ConicBundle and CSDP libraries are
unchanged. At the Phase~1 time limit, Phase~2 uses the latest saved
reformulation. If none is available, the run retains only its initial
feasible solution and has no finite lower bound.
If an SDP oracle evaluation fails, the implementation uses
the last completed finite evaluation, when available, to construct the
reformulation.

Every retained Phase~2 product has full McCormick links. Gurobi uses
seed zero and relative and absolute gap tolerances of $10^{-6}$.
The convex reformulations disable quadratic prelinearization, while direct
Gurobi uses nonconvex mode with automatic prelinearization. Other Gurobi
parameters use their defaults.

We assess optimality certification and feasible-solution quality using
the best verified feasible objective $U$ and final lower bound $L$, both
in the internal minimization convention. The percentage gap is
$100\max\{0,U-L\}/\max\{1,|U|\}$. An instance passes the optimality test
if $U-L<1-10^{-4}$ for integer-valued objectives, or
$U-L<\max\{10^{-6},10^{-6}|U|\}$ for semi-assignment instances.
The ``Certified'' counts report how many instances pass this test.

Incumbents are verified in the original formulation. Final solver bounds,
including CB coefficient-omission corrections, are checked against verified
feasible values from this study and earlier runs; earlier values are used
only for validation. One CB run fails this check. Its bound, gap, and
certification are excluded, while its measured time and verified feasible
objective are retained. Online Resource~1 identifies the case.

\paragraph{AI assistance.}
AI tools assisted with preparing auxiliary scripts for generating tables.
The authors implemented the optimization algorithms and reviewed the
AI-assisted scripts and resulting tables.

\subsection{Certification and computational cost}

\Cref{tab:final-summary} summarizes the complete benchmark. MIQCR--FR
certifies 296 of 431 instances, compared with 156 for MIQCR-CB and 187 for
direct Gurobi. Its mean total time is 1,455.05 seconds, compared with
2,203.08 seconds for direct Gurobi, a reduction of 34.0\% across all
instances.
\begin{table}[ht]\centering\small\setlength{\tabcolsep}{3pt}
\caption{Results for all 431 instances. Times are in seconds; reported times and gaps are arithmetic means.}\label{tab:final-summary}
\begin{tabular}{llrrrrr}\toprule
Family & Method & Certified & P1 & P2 & Total & Gap (\%) \\
\midrule
QAP & MIQCR--FR & 42/61 & 35.90 & 1,619.84 & 1,655.74 & 2.15 \\
 & CB & 22/61 & 724.56 & 1,961.18 & 2,685.74 & 82.34 \\
 & Gurobi & 25/61 & 0.00 & 2,195.63 & 2,195.63 & 38.22 \\
\addlinespace
Cycle cover & MIQCR--FR & 84/102 & 62.00 & 904.51 & 966.51 & 6.04 \\
 & CB & 52/102 & 651.80 & 1,505.76 & 2,098.51 & 1,406.11 \\
 & Gurobi & 91/102 & 0.00 & 728.98 & 728.98 & 7.24 \\
\addlinespace
Semi-assignment & MIQCR--FR & 15/30 & 1.55 & 1,838.66 & 1,840.21 & 4.53 \\
 & CB & 15/30 & 645.08 & 1,387.09 & 2,032.17 & 4.93 \\
 & Gurobi & 6/30 & 0.00 & 2,902.69 & 2,902.69 & 37.42 \\
\addlinespace
$k$-cluster & MIQCR--FR & 57/90 & 3.21 & 1,710.08 & 1,713.29 & 0.32 \\
 & CB & 31/90 & 34.85 & 2,496.87 & 2,531.72 & 0.90 \\
 & Gurobi & 0/90 & 0.00 & 3,606.89 & 3,606.89 & 83.63 \\
\addlinespace
Partitioning & MIQCR--FR & 25/72 & 366.45 & 2,346.50 & 2,712.94 & 3.72 \\
 & CB & 6/72 & 775.78 & 2,573.18 & 3,348.96 & 2,593,404.03 \\
 & Gurobi & 17/72 & 0.00 & 2,798.95 & 2,798.95 & 38.07 \\
\addlinespace
QSPP & MIQCR--FR & 73/76 & 56.97 & 243.13 & 300.09 & 0.82 \\
 & CB & 30/76 & 725.76 & 1,689.02 & 2,259.21 & 129.89 \\
 & Gurobi & 48/76 & 0.00 & 1,684.37 & 1,684.37 & 34.05 \\
\addlinespace
\bottomrule\end{tabular}
\par\smallskip
\begin{minipage}{\linewidth}\footnotesize\raggedright
P2 and gap means include all instances except for CB: four cycle-cover and
seven QSPP instances have no Phase~2 model or finite gap; one QAP gap is
excluded by validation. Total-time means include all runs.
\end{minipage}
\end{table}
\FloatBarrier

The largest certification differences occur on $k$-cluster, QSPP, and QAP.
On $k$-cluster, MIQCR--FR certifies 57 of 90 instances, whereas direct
Gurobi certifies none; the mean final gaps are 0.32\% and 83.63\%,
respectively. On QSPP, MIQCR--FR certifies 73 of 76 instances, compared
with 48 for direct Gurobi, and reduces mean total time from 1,684.37 to
300.09 seconds. QAP certifications increase from 25 to 42.

The remaining families give a more mixed picture. Semi-assignment
certifications increase from 6 to 15, but partitioning improves only from
17 to 25, with similar all-instance mean times. Direct Gurobi performs
better on the tested sparse cycle-cover instances: it certifies 91 instances against 84 for
MIQCR--FR, with mean total times of 728.98 and 966.51 seconds,
respectively. These cycle-cover instances have sparse quadratic
objectives. For such problems, convex reformulation can introduce a dense
quadratic objective and additional product variables, reducing the
advantage of stronger relaxation bounds.

More certifications need not imply faster solution: on the 13 partitioning cases certified
by both MIQCR--FR and direct Gurobi, the median ratio of their total
times is 30.25, in favor of direct Gurobi.

Phase~1 is relatively inexpensive for MIQCR--FR on most families: its
mean cost ranges from 1.55 seconds for semi-assignment to 62.00 seconds
for cycle cover. Partitioning is the exception, at 366.45 seconds.
MIQCR--FR constructs a Phase~2 model for every instance, including the
28 cases in which ADMM reaches its time or iteration limit.

MIQCR-CB reaches Phase~2 on 420 of 431 instances, including 221 runs
that use a saved reformulation after the Phase~1 time limit. Its mean total time is 2,504.69 seconds. Both methods certify 15
semi-assignment cases. On $k$-cluster, CB has a relatively short Phase~1
but certifies 31 cases against 57 for MIQCR--FR. Thus the differences
are not explained solely by the cost of the semidefinite phase.

Very weak CB bounds produce large mean gaps, particularly
on partitioning, so these averages should be read alongside certification
counts and feasible objectives.

\subsection{Quality of feasible solutions}

Certification and gap alone do not distinguish a better feasible solution
from a stronger lower bound. For each instance,
we compare the verified feasible objective values returned by the three
methods. \Cref{tab:final-incumbents} counts how often each method attains
the smallest value, allowing ties within a numerical tolerance. This best
observed value need not be optimal.
\begin{table}[ht]\centering\small\caption{Matches to the best observed feasible objective, with ties allowed.}\label{tab:final-incumbents}\begin{tabular}{lrrr}\toprule
Family & MIQCR--FR & CB & Gurobi \\
\midrule
QAP & 53/61 & 34/61 & 38/61 \\
Cycle cover & 88/102 & 61/102 & 100/102 \\
Semi-assignment & 21/30 & 22/30 & 26/30 \\
$k$-cluster & 90/90 & 80/90 & 37/90 \\
Partitioning & 67/72 & 21/72 & 44/72 \\
QSPP & 74/76 & 43/76 & 70/76 \\
\bottomrule\end{tabular}\end{table}
\FloatBarrier

MIQCR--FR matches the best observed solution on all 90 $k$-cluster
instances, compared with 37 for direct Gurobi. It also obtains more matches on QAP, partitioning, and QSPP, so its gains
include feasible-solution quality as well as certification. Direct Gurobi finds more matches on cycle cover and semi-assignment,
despite fewer certifications on the latter. This illustrates why bounds
and incumbents must be evaluated separately. Runs with no
validated improvement retain the common initial feasible solution; these
values are identified in Online Resource~1.

These results assess the complete MIQCR--FR method; they do not isolate
the contributions of facial reduction, relaxation strength, and
continuous-bound optimization.

\section{Conclusion}

We established MIQCR--FR, a framework for exact convex reformulation from facially
reduced semidefinite relaxations of equality-constrained binary quadratic
problems. A key structural property connects this framework to standard
MIQCR: retaining the original linear equalities and their products with
each binary variable permits recovery of the original dual multipliers
from every feasible point of the reduced dual, without changing its bound.
Thus facial reduction can reduce the semidefinite matrix order while
preserving access to the standard MIQCR construction.

For approximate multiplier estimates, we established three complementary
constructions. A spectral shift restores dual feasibility with an explicit
change in the dual value. Optimizing that value within a restricted
multiplier family yields an efficiently solvable trust-region subproblem.
Full dual feasibility is sufficient but not necessary for convex
reformulation. By enforcing convexity directly, continuous-bound
optimization attains the strongest continuous MIQCR bound within the
chosen multiplier family through a single reduced eigenvalue calculation.
This is the construction used in our ADMM implementation of MIQCR--FR.
Together, these results
provide both dual certificates and direct convex reformulations from
approximate multiplier estimates.

On 431 instances from six problem families, MIQCR--FR certifies 296
optima, compared with 156 for MIQCR-CB and 187 for direct Gurobi, under
the common time budget and numerical validation criteria in
\cref{sec:computational-design}. Its mean total time is 34.0\% lower than
that of direct Gurobi, with the largest certification gains on quadratic
assignment, shortest path, and $k$-cluster.
These results demonstrate the practical value of combining facially
reduced semidefinite optimization with exact convex reformulation.

\section*{Statements and Declarations}
\subsection*{Funding}
Research supported by the Air Force Office of Scientific Research under
award number FA9550-23-1-0508.

\subsection*{Competing interests}
The authors have no relevant financial or non-financial interests to disclose.

\subsection*{Data and code availability}
Instance-level results for the study in \cref{sec:computational-design}
are provided in Online Resource~1. The implementation, instance manifest,
data provenance, solver settings, and machine-readable results are
publicly available at \url{https://github.com/haohu-code/MIQCR-FR}.

\subsection*{Supplementary information}
Online Resource~1 contains the instance-level numerical results for all
431 instances and 1,293 runs, including timings, feasible objectives,
lower bounds, gaps, and solution statuses (Tables S1--S6).

\end{document}


\maketitle
\begin{center}
Corresponding author: Hao Hu\\
School of Mathematical and Statistical Sciences\\
Clemson University, Clemson, SC, USA\\
\href{mailto:hhu2@clemson.edu}{\texttt{hhu2@clemson.edu}}
\end{center}

This supplement reports the complete instance-level results for the
three-method comparison in the accompanying paper. It includes all
431 instances and 1,293 runs, including limits and numerical failures.
The paper contains the experimental protocol, aggregate comparisons,
and their interpretation.

The experiments follow the protocol in Section~4 of the main paper.
At the CB Phase~1 time limit, Phase~2 uses the latest saved reformulation;
11 runs have none.

For the best verified feasible objective $U$ and final lower bound $L$, the reported percentage gap is
$100\max\{0,U-L\}/\max\{1,|U|\}$ when both bounds are valid and finite.

The optimality test and numerical convexity safeguard are specified in
Section~4.3 of the main paper.
Incumbents are verified in the original formulation.
Final bounds are checked against the best verified incumbent across methods
and previously verified feasible values used only for validation.
One CB run, \texttt{nug25}, fails bound validation: its returned lower bound
exceeds a verified feasible objective. Its bound, gap, and certification
are excluded; its measured time and verified incumbent are retained.

\clearpage
\singlespacing
\begingroup
\renewcommand{\scriptsize}{\fontsize{7}{8}\selectfont}
\renewcommand{\small}{\fontsize{9}{10.8}\selectfont}
\section{Instance-level computational results}\label{app:instance-results}
Tables~\ref{tab:instances-qap}--\ref{tab:instances-qspp} list all
431 instances and 1,293 runs. MIQCR--FR denotes the ADMM implementation
of our facial-reduction framework, and CB denotes MIQCR-CB. P1 is the complete reformulation time, P2 the subsequent integer
optimization time, and Total their sum. All times are in seconds.
$U$ is the best verified feasible objective (upper bound), and $L$ is the
final lower bound. All objectives use the internal
minimization convention; $k$-cluster values are
negative cluster profits. Bold upper bounds match the best observed value
across methods. A dagger marks an unchanged initial solution when no validated
improvement is available.
Status codes are O (native optimal), T (Phase-2 time limit),
P1 (Phase-1 limit without an exported model), and V (failed numerical
validation). A star indicates that the run passes the optimality test in
Section~4.3 of the main paper.
An infinite gap means no finite bound was obtained; a dash denotes an
absent time or an excluded bound or gap. A displayed zero gap alone does not indicate optimality,
because percentages are rounded for display.

\begingroup\scriptsize\setlength{\tabcolsep}{2pt}
\begin{longtable}{ll l rrrrrr l}
\caption{QAP: all instances; times in seconds.}\label{tab:instances-qap}\\
\toprule
Instance & $n$ & Method & P1 & P2 & Total & $U$ & $L$ & Gap (\%) & Status \\
\midrule
\endfirsthead
\toprule
Instance & $n$ & Method & P1 & P2 & Total & $U$ & $L$ & Gap (\%) & Status \\
\midrule
\endhead
\midrule\multicolumn{10}{r}{Continued on next page}\\\endfoot
\bottomrule\endlastfoot
chr12a & 144 & MIQCR--FR & 0.85 & 0.35 & 1.20 & \textbf{9,552} & 9,552.00 & 0.00 & O$^{*}$ \\*
 &  & CB & 50.63 & 15.40 & 66.04 & \textbf{9,552} & 9,552.00 & 0.00 & O$^{*}$ \\*
 &  & Gurobi & 0.00 & 1.94 & 1.95 & \textbf{9,552} & 9,552.00 & 0.00 & O$^{*}$ \\
\addlinespace[2pt]
chr12b & 144 & MIQCR--FR & 0.83 & 0.43 & 1.26 & \textbf{9,742} & 9,742.00 & 0.00 & O$^{*}$ \\*
 &  & CB & 460.10 & 0.88 & 460.98 & \textbf{9,742} & 9,741.99 & 0.00 & O$^{*}$ \\*
 &  & Gurobi & 0.00 & 0.65 & 0.65 & \textbf{9,742} & 9,742.00 & 0.00 & O$^{*}$ \\
\addlinespace[2pt]
chr12c & 144 & MIQCR--FR & 1.17 & 0.41 & 1.58 & \textbf{11,156} & 11,156.00 & 0.00 & O$^{*}$ \\*
 &  & CB & 901.35 & 1.24 & 902.59 & \textbf{11,156} & 11,156.00 & 0.00 & O$^{*}$ \\
\end{longtable}\endgroup
\clearpage
\begingroup\scriptsize\setlength{\tabcolsep}{2pt}
\begin{longtable}{ll l rrrrrr l}
\caption{Cycle cover: all instances; times in seconds.}\label{tab:instances-cycle_cover}\\
\toprule
Instance & $n$ & Method & P1 & P2 & Total & $U$ & $L$ & Gap (\%) & Status \\
\midrule
\endfirsthead
\toprule
Instance & $n$ & Method & P1 & P2 & Total & $U$ & $L$ & Gap (\%) & Status \\
\midrule
\endhead
\midrule\multicolumn{10}{r}{Continued on next page}\\\endfoot
\bottomrule\endlastfoot
qcc\_ER\_1 & 119 & MIQCR--FR & 3.09 & 0.63 & 3.72 & \textbf{319} & 319.00 & 0.00 & O$^{*}$ \\*
 &  & CB & 89.45 & 0.77 & 90.22 & \textbf{319} & 319.00 & 0.00 & O$^{*}$ \\*
 &  & Gurobi & 0.00 & 0.47 & 0.47 & \textbf{319} & 319.00 & 0.00 & O$^{*}$ \\
\addlinespace[2pt]
qcc\_ER\_11 & 195 & MIQCR--FR & 1.75 & 27.85 & 29.60 & \textbf{236} & 236.00 & 0.00 & O$^{*}$ \\*
 &  & CB & 900.62 & 38.67 & 939.29 & \textbf{236} & 236.00 & 0.00 & O$^{*}$ \\*
 &  & Gurobi & 0.00 & 13.12 & 13.12 & \textbf{236} & 236.00 & 0.00 & O$^{*}$ \\
\addlinespace[2pt]
qcc\_ER\_12 & 327 & MIQCR--FR & 8.26 & 3,591.96 & 3,600.23 & 206 & 177.51 & 13.83 & T \\*
 &  & CB & 900.91 & 2,699.57 & 3,600.48 & 260 & -43.67 & 116.80 & T \\
\end{longtable}\endgroup
\clearpage
\begingroup\scriptsize\setlength{\tabcolsep}{2pt}
\begin{longtable}{ll l rrrrrr l}
\caption{Semi-assignment: all instances; times in seconds.}\label{tab:instances-semi_assignment}\\
\toprule
Instance & $n$ & Method & P1 & P2 & Total & $U$ & $L$ & Gap (\%) & Status \\
\midrule
\endfirsthead
\toprule
Instance & $n$ & Method & P1 & P2 & Total & $U$ & $L$ & Gap (\%) & Status \\
\midrule
\endhead
\midrule\multicolumn{10}{r}{Continued on next page}\\\endfoot
\bottomrule\endlastfoot
qsap\_10\_3 & 240 & MIQCR--FR & 2.51 & 3,597.92 & 3,600.43 & \textbf{-6,027.7354} & -6,870.30 & 13.98 & T \\*
 &  & CB & 900.71 & 2,699.53 & 3,600.24 & -6,012.1274 & -6,920.61 & 15.11 & T \\*
 &  & Gurobi & 0.00 & 3,603.18 & 3,603.18 & -5,782.7011 & -10,502.61 & 81.62 & T \\
\addlinespace[2pt]
qsap\_12\_3 & 140 & MIQCR--FR & 1.09 & 89.15 & 90.24 & \textbf{-6,127.6362} & -6,127.64 & 0.00 & O$^{*}$ \\*
 &  & CB & 900.50 & 87.78 & 988.28 & \textbf{-6,127.6362} & -6,127.64 & 0.00 & O$^{*}$ \\*
 &  & Gurobi & 0.00 & 3,606.46 & 3,606.46 & \textbf{-6,127.6362} & -7,849.27 & 28.10 & T \\
\addlinespace[2pt]
qsap\_12\_4 & 175 & MIQCR--FR & 1.71 & 3,598.48 & 3,600.19 & \textbf{-6,606.6188} & -6,713.54 & 1.62 & T \\*
 &  & CB & 900.46 & 2,699.72 & 3,600.19 & \textbf{-6,606.6188} & -6,735.22 & 1.95 & T \\
\end{longtable}\endgroup
\clearpage
\begingroup\scriptsize\setlength{\tabcolsep}{2pt}
\begin{longtable}{ll l rrrrrr l}
\caption{$k$-cluster: all instances; times in seconds.}\label{tab:instances-kcluster}\\
\toprule
Instance & $n$ & Method & P1 & P2 & Total & $U$ & $L$ & Gap (\%) & Status \\
\midrule
\endfirsthead
\toprule
Instance & $n$ & Method & P1 & P2 & Total & $U$ & $L$ & Gap (\%) & Status \\
\midrule
\endhead
\midrule\multicolumn{10}{r}{Continued on next page}\\\endfoot
\bottomrule\endlastfoot
kcluster120\_025\_30\_1 & 120 & MIQCR--FR & 0.99 & 724.17 & 725.16 & \textbf{-198} & -198.00 & 0.00 & O$^{*}$ \\*
 &  & CB & 57.94 & 3,542.16 & 3,600.09 & \textbf{-198} & -201.74 & 1.89 & T \\*
 &  & Gurobi & 0.00 & 3,600.37 & 3,600.37 & -186 & -348.00 & 87.10 & T \\
\addlinespace[2pt]
kcluster120\_025\_30\_2 & 120 & MIQCR--FR & 1.13 & 3,017.68 & 3,018.81 & \textbf{-199} & -199.00 & 0.00 & O$^{*}$ \\*
 &  & CB & 57.48 & 3,542.63 & 3,600.11 & -198 & -204.89 & 3.48 & T \\*
 &  & Gurobi & 0.00 & 3,600.32 & 3,600.32 & -194 & -343.00 & 76.80 & T \\
\addlinespace[2pt]
kcluster120\_025\_30\_3 & 120 & MIQCR--FR & 0.99 & 970.00 & 970.99 & \textbf{-197} & -197.00 & 0.00 & O$^{*}$ \\*
 &  & CB & 55.47 & 3,544.63 & 3,600.10 & \textbf{-197} & -201.58 & 2.33 & T \\
\end{longtable}\endgroup
\clearpage
\begingroup\scriptsize\setlength{\tabcolsep}{2pt}
\begin{longtable}{ll l rrrrrr l}
\caption{Partitioning: all instances; times in seconds.}\label{tab:instances-partition}\\
\toprule
Instance & $n$ & Method & P1 & P2 & Total & $U$ & $L$ & Gap (\%) & Status \\
\midrule
\endfirsthead
\toprule
Instance & $n$ & Method & P1 & P2 & Total & $U$ & $L$ & Gap (\%) & Status \\
\midrule
\endhead
\midrule\multicolumn{10}{r}{Continued on next page}\\\endfoot
\bottomrule\endlastfoot
anna\_k3 & 414 & MIQCR--FR & 202.79 & 3,397.42 & 3,600.21 & \textbf{129} & 125.33 & 2.84 & T \\*
 &  & CB & 900.67 & 2,699.48 & 3,600.15 & 165 & 115.21 & 30.17 & T \\*
 &  & Gurobi & 0.00 & 3.93 & 3.93 & \textbf{129} & 129.00 & 0.00 & O$^{*}$ \\
\addlinespace[2pt]
anna\_k4 & 552 & MIQCR--FR & 315.04 & 3,285.23 & 3,600.27 & \textbf{162} & 155.50 & 4.01 & T \\*
 &  & CB & 900.86 & 2,699.27 & 3,600.13 & 359 & -104,894,571.16 & 29,218,643.50 & T \\*
 &  & Gurobi & 0.00 & 63.36 & 63.36 & \textbf{162} & 162.00 & 0.00 & O$^{*}$ \\
\addlinespace[2pt]
david\_k3 & 261 & MIQCR--FR & 39.97 & 477.50 & 517.47 & \textbf{126} & 126.00 & 0.00 & O$^{*}$ \\*
 &  & CB & 900.44 & 2,699.69 & 3,600.13 & 127 & 114.50 & 9.84 & T \\
\end{longtable}\endgroup
\clearpage
\begingroup\scriptsize\setlength{\tabcolsep}{2pt}
\begin{longtable}{ll l rrrrrr l}
\caption{QSPP: all instances; times in seconds.}\label{tab:instances-qspp}\\
\toprule
Instance & $n$ & Method & P1 & P2 & Total & $U$ & $L$ & Gap (\%) & Status \\
\midrule
\endfirsthead
\toprule
Instance & $n$ & Method & P1 & P2 & Total & $U$ & $L$ & Gap (\%) & Status \\
\midrule
\endhead
\midrule\multicolumn{10}{r}{Continued on next page}\\\endfoot
\bottomrule\endlastfoot
GRID1\_square\_01 & 220 & MIQCR--FR & 1.81 & 3.25 & 5.06 & \textbf{206} & 206.00 & 0.00 & O$^{*}$ \\*
 &  & CB & 77.05 & 83.33 & 160.38 & \textbf{206} & 206.00 & 0.00 & O$^{*}$ \\*
 &  & Gurobi & 0.00 & 2.45 & 2.45 & \textbf{206} & 206.00 & 0.00 & O$^{*}$ \\
\addlinespace[2pt]
GRID1\_square\_02 & 220 & MIQCR--FR & 5.91 & 0.93 & 6.83 & \textbf{181} & 181.00 & 0.00 & O$^{*}$ \\*
 &  & CB & 95.68 & 28.70 & 124.38 & \textbf{181} & 181.00 & 0.00 & O$^{*}$ \\*
 &  & Gurobi & 0.00 & 1.68 & 1.68 & \textbf{181} & 181.00 & 0.00 & O$^{*}$ \\
\addlinespace[2pt]
GRID1\_square\_03 & 220 & MIQCR--FR & 1.09 & 0.94 & 2.02 & \textbf{1,375} & 1,375.00 & 0.00 & O$^{*}$ \\*
 &  & CB & 900.58 & 4.56 & 905.14 & \textbf{1,375} & 1,375.00 & 0.00 & O$^{*}$ \\
\end{longtable}\endgroup
\endgroup